\documentclass[11pt]{article}
\usepackage{amsmath,amsthm}
\usepackage{xcolor}
\usepackage{amssymb, amsfonts}
\usepackage{graphics}
\usepackage{epsfig,graphicx, natbib,bbm,bm}
\usepackage[ansinew]{inputenc}
\usepackage{enumerate}

\usepackage{changes}
\definechangesauthor[name={Soeren}, color=red]{soe}
\definechangesauthor[name={Luis}, color=blue]{lae}

\newtheorem{theorem}{Theorem}[section]

\newtheorem{lemma}[theorem]{Lemma}

\newtheorem{rmk}{Remark}[section]

\newcommand{\argmax}{\operatornamewithlimits{argmax}}

\DeclareMathOperator{\sgn}{sgn}

\numberwithin{equation}{section}

\def\R{\mathbb{R}}

\begin{document}

\title{A Class of Solvable Multidimensional Stopping Problems in the Presence of Knightian Uncertainty}

\author{Luis H. R. Alvarez E.\thanks{Department of Accounting and Finance, Turku School of Economics,
FIN-20014 University of Turku, Finland, E-mail: lhralv@utu.fi}\quad
S\"oren Christensen\thanks{Mathematisches Seminar,
Christian-Albrechts-Universität zu Kiel,
Ludewig-Meyn-Str. 4,
D-24098 Kiel,
Germany,  E-mail: christensen@math.uni-kiel.de}}
\maketitle

\abstract{We investigate the impact of Knightian uncertainty on the optimal timing policy of an ambiguity averse decision maker in the case where the underlying factor dynamics follow a multidimensional Brownian motion and the exercise payoff depends on either a linear combination of the factors or the radial part of the driving factor dynamics. We present a general characterization of the value of the optimal timing policy and the worst case measure in terms of a family of an explicitly identified excessive functions generating an appropriate class of supermartingales.
In line with previous findings based on linear diffusions, we find that ambiguity accelerates timing in comparison with the unambiguous setting. Somewhat surprisingly, we find that ambiguity may result into stationarity in models which typically do not possess stationary behavior. In this way, our results indicate that ambiguity may act as a stabilizing mechanism.}\\

\noindent{\bf AMS Subject Classification:} 60J60, 60G40, 62L15, 91G80 \\

\noindent{\bf Keywords:} $\kappa$-ambiguity, multidimensional Brownian motion, diffusion processes, Bessel processes.

\thispagestyle{empty} \clearpage \setcounter{page}{1}

\section{Introduction}

Gaussian processes and, more precisely, Brownian motion plays a prominent role in modeling factor dynamics in standard
financial models considering the optimal timing of irreversible decisions in the presence of uncertainty. In the benchmark
setting all the uncertainty affecting the decision is summarized into a single probability measure describing completely
the probabilistic structure of the underlying intertemporally fluctuating factor dynamics. However, as originally pointed out in
\cite{Kn21}, in reality there are circumstances where a decision maker faces unmeasurable uncertainty on the plausibility
or credibility of a particular probability measure (so-called {\em Knightian uncertainty}). In such a case a decision maker
may have to make a decision based on several or even a continuum of different measures describing the probabilistic structure of the
alternative states of the world.

Ambiguity was first rigorously axiomatized based on the pioneering work by \citet{Kn21}
in a atemporal multiple priors setting in
by \citet{GiSch89} (for further refinements, see also \citet{Be02}, \citet{Kli_et_al_05},
\citet{Ma_et_al_06} and \citet{NiOz06}). The atemproal axiomatization was subsequently extended into an intertemporal
recursive multiple priors setting by, among others, \cite{EpWa94}, \citet{ChEp02}, \cite{EpMi03}, and \citet{EpSch03}.
The impact of ambiguity on optimal timing decisions was originally studied in
\cite{NiOz04} in a job search model. They subsequently extended their original analysis in \cite{NiOz07} by considering the impact of Knightian uncertainty on the optimal timing decisions of irreversible investment opportunities in a continuous time model based on geometric Brownian motion. \cite{Al07} focused on the impact of Knightian uncertainty on monotone one-sided stopping problems and expressed the value as well as the optimality conditions for the stopping boundaries in terms of the minimal excessive mappings of the underlying diffusion under the worst case measure.
\cite{Ri2009}, in turn, analyzed discrete time optimal stopping problems in the presence of ambiguity aversion and developed a general minmax martingale approach for solving the considered problems (see also \cite{MiWa_11} for an analysis of the problem for a general discrete time Feller-continuous Markov processes). The approach developed in \cite{Ri2009} was subsequently extended to a continuous time setting in \cite{ChRi2013}. In \cite{ChRi2013},  the value of the optimal policy was proven to be the smallest right continuous $g$-martingale dominating the exercise payoff process.
\cite{Chr13} investigated the optimal stopping of linear diffusions by ambiguity averse decision makers in the presence of Knightian uncertainty and identified explicitly the minimal excessive mappings generating the worst case measure as well as the appropriate class of supermartingales needed for the characterization of the value of the optimal policy. {\cite{EpJi2019} investigated optimal learning in the case where the underlying driving Brownian motion is subject to drift ambiguity}. More recently, \cite{AlCh2019} extended the approach developed in \cite{Chr13} to a multidimensional setting and investigated the impact of Knightian uncertainty on the optimal timing policies of ambiguity averse investors in the case where the exercise payoff is positively homogeneous and the underlying diffusion is a two-dimensional geometric Brownian motion. They found that in a multidimensional case, ambiguity does not only affect the optimal policy by altering the rate at which the underlying processes are expected to grow, it also impacts the rate at which the problem is discounted.

Given the findings in \cite{AlCh2019}, our objective in this paper is to analyze the impact of Knightian uncertainty on the optimal timing policy of an ambiguity averse decision maker in the case where the underlying follows a multidimensional Brownian motion. We study the general stopping problem and identify two special cases under which the problem can be explicitly solved by reducing the dimensionality of the problem and then utilizing the approach developed in \cite{Chr13}. We characterize the value and optimal timing policies as the smallest majorizing element of the exercise payoff in a parameterized function space.  Our results demonstrate that Knightian uncertainty does not only accelerate the optimal timing policy in comparison with the unambiguous benchmark case, it also may result into stationary behavior to the controlled system even when the underlying system does not possess a long run stationary distribution. This observation illustrates how ambiguity may in some cases have a nontrivial impact on the stochastic dynamics of the underlying processes under the worst case measure.

The contents of this paper are as follows. In Section 2 we present the underlying stochastic dynamics, state the considered class of optimal stopping problems and state a characterization of the impact of ambiguity on the optimal timing policy and its value. In Section 3 we focus on payoffs depending on linear combinations of the driving factors. In Section 4 we then focus on radially symmetric payoffs. Finally, Section 5 concludes our study.

\section{Underlying Dynamics and Problem Setting}

Let $\mathbf W$ be $d$-dimensional standard Brownian motion under the measure $\mathbb{P}$ and assume that $d\geq 2$.
As usually in models subject to Knightian uncertainty, let the degree of ambiguity $\kappa>0$ be given and denote by $\mathcal{P}^\kappa$ the set of all probability measures, that are equivalent to $\mathbb{P}$ with density process of the form
$$
\mathcal{M}_t^{{\bm \theta}}=e^{-\int_0^t {\bm \theta}_s d\mathbf{W}_s - \frac{1}{2}\int_0^t \|{\bm \theta_s}\|^2 ds}
$$
for a progressively measurable process $\{{\bm \theta}_t\}_{t\geq 0}$ satisfying the inequality $\|{\bm \theta}_t \|^2\leq \kappa^2$ for all $t\geq 0$. That is, we assume that the density generator processes satisfy the inequality $\sum_{i=1}^d\theta_{it}^2\leq \kappa^2$ for all $t\geq 0$.

Assume now that $\mathbf{X}_t=\mathbf{x}+\mathbf{W}_t$ denotes the underlying diffusion under the measure $\mathbb{P}$. Our objective is now to consider the following optimal stopping problem
\begin{align}
V_\kappa(\bm x) = \sup_{\tau\in \mathcal{T}}\inf_{\mathbb{Q}^{{\bm \theta}}\in \mathcal{P}^\kappa}\mathbb{E}_{\bm x}^{{\mathbb{Q}^{{\bm \theta}}}}\left[e^{-r\tau}F(\mathbf{X}_\tau)\mathbbm{1}_{\{\tau < \infty\}}\right],\label{stopping}
\end{align}
where $F:\mathbb{R}^d\mapsto \mathbb{R}$ is a measurable function which will be specified below in the two cases considered in this paper. As usually, we denote by $C_\kappa=\{\bm x\in \mathbb{R}^d:V_\kappa(\bm x)>F(\bm x)\}$ the continuation region where stopping is suboptimal and by $\Gamma_\kappa=\{\bm x\in \mathbb{R}^d:V_\kappa(\bm x)=F(\bm x)\}$ the stopping region. The specification of the considered stopping problem results into the following lemma characterizing the impact of ambiguity on the optimal stopping policy and its value in a general setting.
\begin{lemma}\label{l1}
Increased ambiguity accelerates optimal timing by decreasing the value of the optimal policy and, thus, shrinking the continuation region where waiting is optimal. Formally, if $\hat{\kappa}>\kappa$ then $V_{\hat{\kappa}}(\bm x)\leq V_{\kappa}(\bm x)$ for all $\bm x\in\mathbb{R}^d$ and $C_{\hat{\kappa}}\subseteq C_{\kappa}$.
\end{lemma}
\begin{proof}
Assume that $\hat{\kappa}>\kappa$. Since $\{{\bm \theta}\in \mathbb{R}^d:\|{\bm \theta} \|^2\leq \kappa^2\}\subset \{{\bm \theta}\in \mathbb{R}^d:\|{\bm \theta} \|^2\leq \hat{\kappa}^2\}$ we notice that
$$
\inf_{\mathbb{Q}^{{\bm \theta}}\in \mathcal{P}^{\hat{\kappa}}}\mathbb{E}_{\bm x}^{{\mathbb{Q}^{{\bm \theta}}}}\left[e^{-r\tau}F(\mathbf{X}_\tau)\mathbbm{1}_{\{\tau < \infty\}}\right]\leq \inf_{\mathbb{Q}^{{\bm \theta}}\in \mathcal{P}^{\kappa}}\mathbb{E}_{\bm x}^{{\mathbb{Q}^{{\bm \theta}}}}\left[e^{-r\tau}F(\mathbf{X}_\tau)\mathbbm{1}_{\{\tau < \infty\}}\right]
$$
implying that $V_{\hat{\kappa}}(\bm x)\leq V_{\kappa}(\bm x)$ for all $\bm x\in\mathbb{R}^d$. Assume now that ${\bm x}\in C_{\hat{\kappa}}$. Since in that case $V_{\kappa}(\bm x)\geq  V_{\hat{\kappa}}(\bm x) > F(\bm x)$ we notice that ${\bm x}\in C_{\kappa}$ as well. Consequently, $C_{\hat{\kappa}}\subseteq C_{\kappa}$ completing the proof of our lemma.
\end{proof}
Lemma \ref{l1} shows that the sign of the relationship between the degree of ambiguity and optimal timing is positive. At the same time, increased ambiguity decreases the value of the optimal stopping policy showing that the highest value is attained in the absence of ambiguity. This mechanism is naturally not that surprising since it essentially states that the larger the set of potentially detrimental outcomes gets, the smaller is the achievable value.

We now notice that under the measure $\mathbb{Q}^{{\bm \theta}}$ defined by the likelihood ratio
$$
\frac{d\mathbb{Q}^{{\bm \theta}}}{d\mathbb{P}}=\mathcal{M}_t^{{\bm \theta}}
$$
we naturally have that
$$
\mathbf{X}_t = \mathbf{x} -\int_{0}^{t}\bm{\theta}_sds + \mathbf{W}_t^{{\bm \theta}},
$$
where $\mathbf{W}_t^{{\bm \theta}}$ denotes $\mathbb{Q}^{{\bm \theta}}$-Brownian motion. Introduce the differential operator associated with the underlying processes $\mathbf{X}$ under the measure $\mathbb{Q}^{{\bm \theta}}\in \mathcal{P}^\kappa$ by
$$
\mathcal{A}^{{\bm \theta}}=\frac{1}{2}\sum_{i=1}^d\frac{\partial^2}{\partial x_i^2}-\sum_{i=1}^d\theta_i\frac{\partial}{\partial x_i}.
$$
For a twice continuously differentiable function $u:\mathbb{R}^d\mapsto \mathbb{R}_+$,
the It{\^o}-Döblin theorem yields that under the measure $\mathbb{Q}^{{\bm \theta}}\in \mathcal{P}^\kappa$
\begin{align}\label{id}
\begin{split}
e^{-rt}u(\mathbf{X}_{t})=u(\mathbf{x}) + \int_0^{t}e^{-rs}\left((\mathcal{A}^{{\bm \theta}}u)(\mathbf{X}_{s})-ru(\mathbf{X}_{s})\right)ds
+\int_0^t e^{-rs}\nabla u(\mathbf{X}_{s})\cdot d\mathbf{W}_s^{{\bm \theta}}.
\end{split}
\end{align}
Now, minimizing $(\mathcal{A}^{{\bm \theta}}u)({\bm x})$ with respect to ${\bm \theta}$ under the condition $\|{\bm \theta}\|^2\leq \kappa^2$ leads to the worst case density generator
\[{\bm\theta}^*_t=\kappa \frac{\nabla u(\mathbf{X}_{t})}{\|\nabla u(\mathbf{X}_{t})\|},\]
where $\|\cdot\|$ denotes the standard Euclidean norm. Assume now that there exists a twice continuously differentiable function $\bar{u}:\mathbb{R}^d\mapsto \mathbb{R}_+$ satisfying the partial differential equation
\begin{align}
\frac{1}{2}(\Delta\bar{u})(\mathbf{x})-\kappa \|\nabla \bar{u}(\mathbf{x})\|-r\bar{u}(x)=0\label{PDE}
\end{align}
on some  $G\subseteq \mathbb{R}^d$.
In that case
\begin{align}\label{subharm}
\begin{split}
e^{-rT}\bar{u}(\mathbf{X}_{T})&=\bar{u}(\mathbf{x}) + \int_0^{T}e^{-rs}\left(\kappa\|\nabla \bar{u}(\mathbf{X}_s)\|-{\bm \theta}_s\cdot\nabla \bar{u}(\mathbf{X}_s)\right)ds
+\int_0^T e^{-rs}\nabla \bar{u}(\mathbf{X}_{s})\cdot d\mathbf{W}_s^{{\bm \theta}}\\
&\geq \bar{u}(\mathbf{x}) +\int_0^T e^{-rs}\nabla \bar{u}(\mathbf{X}_{s})\cdot d\mathbf{W}_s^{{\bm \theta}},
\end{split}
\end{align}
where $T=t\wedge \inf\{t\geq 0:\mathbf{X}_t\not \in A\}$ and $A\subseteq G$ is open with compact closure in $G$. Taking expectations result in
\begin{align*}
\mathbb{E}^{\mathbb{Q}^{{\bm \theta}}}_{\mathbf{x}}\left[e^{-rT}\bar{u}(\mathbf{X}_{T})\right]\geq \bar{u}(\mathbf{x})
\end{align*}
with identity only when ${\bm \theta}^\ast={\bm \theta}$.
Unfortunately, solving the partial differential equation \eqref{PDE} explicitly is typically impossible. Fortunately, there are two cases where dimension reduction techniques apply and permit the transformation of the original multidimensional problem into a solvable one-dimensional setting. We will focus on these problems in the following sections.

\section{Payoff Depending on a Linear Combination of Factors}
Linear combinations of independent normally distributed random variables are normally distributed. On the other hand, linear combinations of independent Brownian motions are continuous martingales and, hence, constitute a time change of Brownian motion. Given these observations, consider now the case where the exercise payoff reads as
\begin{align}\label{linearpayoff}
F({\bf x})= \hat{F}\left(\mathbf{a}^T\mathbf{x}\right) = \hat{F}\left(\sum_{i=1}^{d}a_ix_i\right),
\end{align}
where $\mathbf{a}\in \mathbb{R}^d$ is a constant parameter vector and $\hat{F}:\mathbb{R}\mapsto \mathbb{R}$ is a measurable function. Focusing now on functions
\[u({\bf x})= h\left(\mathbf{a}^T\mathbf{x}\right) \]
results into the worst case prior characterized by the density generator
\[{\bm\theta}^*=\kappa\sgn(h'(\mathbf{a}^T\mathbf{x})) \frac{{\bm a}}{\|{\bm a}\|}.\]
In this case, solving
\[(\mathcal{A}^{{\bm \theta}^*}u)({\bm x})=ru({\bm x})\]
results into solving
\[\frac{1}{2}\|\mathbf{a}\|^2h''(\mathbf{a}^T\mathbf{x})-\kappa\| \mathbf{a} \| h'(\mathbf{a}^T\mathbf{x})-r h(\mathbf{a}^T\mathbf{x})=0\]
on $\{{\bm x}:h'(\mathbf{a}^T\mathbf{x})\geq 0\}$ and
\[\frac{1}{2}\|\mathbf{a}\|^2h''(\mathbf{a}^T\mathbf{x})+\kappa\| \mathbf{a} \| h'(\mathbf{a}^T\mathbf{x})-r h(\mathbf{a}^T\mathbf{x})=0\]
on on $\{{\bm x}:h'(\mathbf{a}^T\mathbf{x})<0\}$. Defining now the constants
\begin{align*}
\psi_{\kappa} &=\frac{\kappa}{\|\mathbf{a}\|}+\sqrt{\frac{\kappa^2}{\|\mathbf{a}\|^2}+\frac{2r}{\|\mathbf{a}\|^2}},\\
\varphi_{\kappa} &=\frac{\kappa}{\|\mathbf{a}\|}-\sqrt{\frac{\kappa^2}{\|\mathbf{a}\|^2}+\frac{2r}{\|\mathbf{a}\|^2}},
\end{align*}
$\hat{\psi}_\kappa=-\varphi_\kappa$, and $\hat{\varphi}_\kappa=-\psi_\kappa$
then shows that
$$
h(y)=c_1 e^{\psi_{\kappa} y} + c_2 e^{\varphi_{\kappa} y}
$$
on $\{y:h'(y)\geq 0\}$ and
$$
h(y)=\hat{c}_1 e^{\hat{\psi}_\kappa y} + \hat{c}_2 e^{\hat{\varphi}_\kappa y}
$$
on $\{y:h'(y)< 0\}$. Given these functions, let $c\in \mathbb{R}$ be an arbitrary reference point and define the twice continuously differentiable and strictly convex function $U_c:\mathbb{R}\mapsto \mathbb{R}$ as $U_c(y)=\max(h_{1c}(y),h_{2c}(y))$, where
\begin{align*}
h_{1c}(y)&= \frac{\psi_{\kappa}}{\psi_{\kappa}-\varphi_{\kappa}}e^{\varphi_{\kappa} (y-c)}-\frac{\varphi_{\kappa}}{\psi_{\kappa}-\varphi_{\kappa}}e^{\psi_{\kappa} (y-c)}\\
h_{2c}(y)&= \frac{\hat{\psi}_{\kappa}}{\hat{\psi}_{\kappa}-\hat{\varphi}_{\kappa}}e^{\hat{\varphi}_{\kappa} (y-c)}-\frac{\hat{\varphi}_{\kappa}}{\hat{\psi}_{\kappa}-\hat{\varphi}_{\kappa}}e^{\hat{\psi}_{\kappa} (y-c)}
\end{align*}
are two mappings satisfying the conditions $h_{1c}(c)=h_{2c}(c)=1$ and $h_{1c}'(c)=h_{2c}'(c)=0$. Therefore, the function $U_c$ constitutes the solution of the boundary value problem
\begin{align*}
&\frac{1}{2}\|\mathbf{a}\|^2U_c''(\mathbf{a}^T\mathbf{x})-\kappa\sgn(\mathbf{a}^T\mathbf{x}-c)\| \mathbf{a} \| U_c'(\mathbf{a}^T\mathbf{x})-r U_c(\mathbf{a}^T\mathbf{x})=0\\
&U_c(c)=1,\quad U_c'(c)=0.
\end{align*}
Analogously, we let
$U_{-\infty}(y)=e^{\psi_{\kappa} y}$ and $U_{\infty}(y)=e^{\hat{\varphi}_\kappa y}$ denote the solutions associated with the extreme cases where $c=-\infty$ or $c=\infty$.
As was demonstrated in \cite{Chr13} these functions generate an useful class of supermartingales for solving optimal stopping problems in the presence of ambiguity.
To see that this is indeed the case in this multidimensional setting as well, we notice by applying the It{\^o}-Döblin theorem to the function $U_c$ that
\begin{align*}
e^{-rT}U_c(\mathbf{a}^T{\bf X}_T) &= U_c\left(\mathbf{a}^T\mathbf{x}\right)+\int_0^Te^{-rt}\left(\kappa\sgn(\mathbf{a}^T\mathbf{X}_t-c)\|\mathbf{a}\|-\mathbf{a}^T{\bm\theta}_t\right)U_c'(\mathbf{a}^T\mathbf{X}_t)dt\\
&+\int_0^T e^{-rt}U_c'(\mathbf{a}^T\mathbf{X}_t)\mathbf{a}^Td\mathbf{W}_t^{{\bm \theta}}.
\end{align*}
Since $-\kappa\|\mathbf{a}\|\leq -\mathbf{a}^T{\bm \theta} \leq \kappa\|\mathbf{a}\|$ for admissible density generators satisfying the condition $\|{\bm \theta}\|^2\leq \kappa^2$, we observe that
$\left(\kappa\sgn(\mathbf{a}^T\mathbf{x}-c)\|\mathbf{a}\|-\mathbf{a}^T{\bm\theta}\right)U_c'(\mathbf{a}^T\mathbf{x})\geq 0$ for all $\mathbf{x}\in \mathbb{R}^d$ and, therefore, that
\begin{align*}
e^{-rT}U_c(\mathbf{a}^T{\bf X}_T) \geq U_c\left(\mathbf{a}^T\mathbf{x}\right)+ \int_0^T e^{-rt}U_c'(\mathbf{a}^T\mathbf{X}_t)\mathbf{a}^Td\mathbf{W}_t^{{\bm \theta}}
\end{align*}
with identity only when ${\bm \theta}_t={\bm \theta}_t^\ast=\kappa\sgn(\mathbf{a}^T\mathbf{X}_t-c).$ Consequently, we notice that in the present case
\begin{align*}
\mathbb{E}^{\mathbb{Q}^{{\bm \theta}}}_{\mathbf{x}}\left[e^{-rT}U_c(\mathbf{a}^T{\bf X}_T)\right] \geq  \mathbb{E}^{\mathbb{Q}^{{\bm \theta}^\ast}}_{\mathbf{x}}\left[e^{-rT}U_c(\mathbf{a}^T{\bf X}_T)\right] = U_c\left(\mathbf{a}^T\mathbf{x}\right)
\end{align*}
for all $\mathbb{Q}^{{\bm \theta}}\in \mathcal{P}^\kappa$. Utilizing standard optional sampling arguments show that the process $\{e^{-rt}U_c(\mathbf{a}^T{\bf X}_t)\}_{t\geq 0}$ is actually a positive $\mathbb{Q}^{{\bm \theta}^\ast}$-martingale and, therefore, a supermartingale.\\

It is also at this point worth pointing out that the process $Y_t=\mathbf{a}^T\mathbf{X}_t$ satisfies the SDE
\begin{align}\label{sde}
dY_t = \mathbf{a}^Td\mathbf{X}_t=-\mathbf{a}^T{\bm \theta}_tdt+\mathbf{a}^Td\mathbf{W}_t^{{\bm \theta}},\quad Y_0=\mathbf{a}^T\mathbf{x}.
\end{align}
Since $-\kappa\|\mathbf{a}\|\leq -\mathbf{a}^T{\bm \theta}_t \leq \kappa\|\mathbf{a}\|$ for admissible density generators satisfying the condition $\|{\bm \theta}_t\|\leq \kappa$ we notice that
\eqref{sde} has a unique strong solution. Especially, under $\mathbb{Q}^{{\bm \theta}^\ast}$ we have
$$
dY_t=-\kappa\|\mathbf{a}\|\sgn(Y_t-c)dt+\mathbf{a}^Td\mathbf{W}_t^{{\bm \theta}^\ast},\quad Y_0=\mathbf{a}^T\mathbf{x},
$$
which is a standard Brownian motion with alternating drift. Interestingly, we observe that while standard Brownian motion does not have a stationary distribution, the controlled process does.
More precisely, for a fixed reference point $c$ the stationary distribution of the controlled diffusion reads as (a {\em Laplace}-distribution)
$$
p(\mathbf{a}^T\mathbf{x})=\frac{\kappa}{\|\mathbf{a}\|}e^{-\frac{2\kappa}{\|\mathbf{a}\|}|\mathbf{a}^T\mathbf{x}-c|}
$$
Moreover, the process $Y_t$ is positively recurrent meaning that hitting times to constant boundaries are almost surely finite.\\

Having characterized the underlying dynamics and the class of harmonic functions resulting into the class of supermartingales needed in the characterization of the value, we now observe that the conditions of Theorem 1 in \cite{Chr13} are satisfied and, therefore, that
we can characterize the value in a semiexplicit form as stated in the following.
\begin{theorem}\label{thm1}
(A) For all $\mathbf{x}\in \mathbb{R}^d$ we have that
\begin{align}\label{program}
V_\kappa(\bm x) = \inf\{\lambda U_c(\mathbf{a}^T\mathbf{x}):c\in[-\infty,\infty],\lambda\in[0,\infty],\lambda  U_c(\mathbf{a}^T\mathbf{x})\geq \hat{F}(\mathbf{a}^T\mathbf{x})\}
\end{align}
and the infimum with respect the reference point $c$ is a minimum.\\
(B) A point $\mathbf{x}\in \mathbb{R}^d$ is in the stopping region $\Gamma=\{\mathbf{x}\in \mathbb{R}^d: V_\kappa(\bm x)=\hat{F}(\mathbf{a}^T\mathbf{x})\}$ if, and only if, there exists a $c\in [-\infty,\infty]$ such that
$$
y_c\in \argmax\left\{\frac{\hat{F}(y)}{U_c(y)}\right\}
$$
and $\mathbf{a}^T\mathbf{x}=y_c$.
\end{theorem}
\begin{proof}
The alleged claims are direct implications of Theorem 1 in \cite{Chr13}.
\end{proof}
Theorem \ref{thm1} extends the findings of Theorem 1 in \cite{Chr13} to the present case. The main reason for the validity of this extension is naturally the fact the even though the process $\mathbf{X}_t$ is multidimensional, the process $\mathbf{a}^T\mathbf{X}_t$ is not and we can, therefore,
analyze the problem in terms of the one-dimensional characteristics. The representation \eqref{program} is naturally useful in the determination of the value and the associated worst case prior since it essentially reduces the analysis of the original problem into the analysis of a ratio with known properties without having to invoke strong smoothness or  regularity conditions. In order to illustrate the usefulness of the finding of Theorem \ref{thm1} we now consider an interesting class of exercise payoffs resulting into an explicitly solvable symmetric setting within this class of models. Our main findings on these problems are summarized in the following.
\begin{theorem}\label{even}
Assume that the exercise payoff $\hat{F}(x)$ is even, that is, that $\hat{F}(x)=\hat{F}(-x)$ for all $x\geq 0$. Then, the ratio $\hat{F}(x)/U_0(x)$ is even as well and if there exists a unique threshold
\begin{align*}
x^\ast =\argmax_{x>0}\left\{\frac{\hat{F}(x)}{U_0(x)}\right\},
\end{align*}
so that $\hat{F}(x)/U_0(x)$ is increasing on $(0,x^\ast)$ and decreasing on $(x^\ast,\infty)$, then the value of the optimal stopping policy $\inf\{t\geq 0:\mathbf{a}^T\mathbf{X}_t\not\in (-x^\ast,x^\ast)\}$ reads as
\begin{align}\label{value}
V_\kappa(\mathbf{x})=\begin{cases}
\hat{F}(\mathbf{a}^T\mathbf{x}),&\mathbf{a}^T\mathbf{x}\not\in (-x^\ast,x^\ast),\\
\frac{\hat{F}(x^\ast)}{U_0(x^\ast)}U_0(\mathbf{a}^T\mathbf{x}),&\mathbf{a}^T\mathbf{x}\in (-x^\ast,x^\ast).
\end{cases}
\end{align}
Moreover, the optimal density generator resulting into the worst case measure is
\[{\bm\theta}_t^*=\kappa\sgn(\mathbf{a}^T\mathbf{X}_t) \frac{{\bm a}}{\|{\bm a}\|}.\]
\end{theorem}
\begin{proof}
We first observe utilizing the identities $\hat{\varphi}_\kappa = -\psi_\kappa$ and
$\hat{\psi}_\kappa=-\varphi_\kappa$ that $U_0(x)=U_0(-x)$ for all $x\geq 0$. Consequently, we notice that the ratio $\hat{F}(x)/U_0(x)$
is even as claimed. Assume now that there exists a unique maximizer $x^\ast >0$ of the ratio $\hat{F}(x)/U_0(x)$ so that $\hat{F}(x)/U_0(x)$ is increasing on $(0,x^\ast)$ and decreasing on $(x^\ast,\infty)$.

Denote now by $\tau^\ast=\inf\{t\geq 0: \mathbf{a}^T \mathbf{X}_t\not\in (-x^\ast,x^\ast)\}$ the first exit time of the process $\mathbf{a}^T\mathbf{X}_t$ from the set $(-x^\ast,x^\ast)$ and by $\hat{V}_\kappa(\mathbf{x})$ the proposed value function \eqref{value}. It is clear that since $\mathbb{Q}^{{\bm \theta}^\ast}\in\mathcal{P}^\kappa$ we have for any admissible stopping time $\tau\in \mathcal{T}$ that
\begin{align*}
\inf_{\mathbb{Q}^{{\bm \theta}}\in \mathcal{P}^\kappa}\mathbb{E}_{\mathbf{x}}^{\mathbb{Q}^{{\bm \theta}}}\left[e^{-r\tau}\hat{F}(\mathbf{a}^T\mathbf{X}_{\tau})\mathbbm{1}_{\{\tau<\infty\}}\right]\leq \mathbb{E}_{\mathbf{x}}^{\mathbb{Q}^{{\bm \theta}^\ast}}\left[e^{-r\tau}\hat{F}(\mathbf{a}^T\mathbf{X}_{\tau})\mathbbm{1}_{\{\tau<\infty\}}\right].
\end{align*}
Consequently, we find that
$$
V_\kappa(\mathbf{x})\leq \sup_{\tau\in \mathcal{T}}\mathbb{E}_{\mathbf{x}}^{\mathbb{Q}^{{\bm \theta}^\ast}}\left[e^{-r\tau}\hat{F}(\mathbf{a}^T\mathbf{X}_{\tau})\mathbbm{1}_{\{\tau<\infty\}}\right].
$$
Consider now the process
$$
\mathcal{M}_t=e^{-rt}U_0(\mathbf{a}^T\mathbf{X}_{t}).
$$
As was shown earlier, $\mathcal{M}_t$ is a positive $\mathbb{Q}^{{\bm \theta}^\ast}$-martingale. Moreover, since the process characterized by the SDE
$$
dY_t=-\kappa\|\mathbf{a}\|\sgn(Y_t)dt+\mathbf{a}^Td\mathbf{W}_t^{{\bm \theta}^\ast},\quad Y_0=\mathbf{a}^T\mathbf{x},
$$
is positively recurrent we know that the first exit time $\tau^\ast=\inf\{t\geq 0: Y_t\not\in (-x^\ast,x^\ast)\}=\inf\{t\geq 0: \mathbf{a}^T \mathbf{X}_t\not\in (-x^\ast,x^\ast)\}$ is $\mathbb{Q}^{{\bm \theta}^\ast}$-almost surely finite. Consequently, the assumed maximality of the ratio $\hat{F}(x^\ast)/U_0(x^\ast)=\hat{F}(-x^\ast)/U_0(-x^\ast)$ guarantees that Theorem 4 of \cite{BeLe1997} applies and we find that (see also \cite{LeUr2007}, \cite{ChIr11}, and \cite{GaLe2011})
$$
\hat{V}_\kappa(\mathbf{x})=\sup_{\tau\in \mathcal{T}}\mathbb{E}_{\mathbf{x}}^{\mathbb{Q}^{{\bm \theta}^\ast}}\left[e^{-r\tau}\hat{F}(\mathbf{a}^T\mathbf{X}_{\tau})\mathbbm{1}_{\{\tau<\infty\}}\right]
$$
proving that $V_\kappa(\mathbf{x})\leq \hat{V}_\kappa(\mathbf{x})$ for all $\mathbf{x}\in \mathbb{R}^d$.
In order to reverse this inequality we first observe that if $\mathbf{x}\in\{\mathbf{x}\in\mathbb{R}^d:\mathbf{a}^T\mathbf{x}\in (-x^\ast,x^\ast)\}$ then we naturally have that
\begin{align*}
V_\kappa(\mathbf{x})&\geq \inf_{\mathbb{Q}^{{\bm \theta}}\in \mathcal{P}^\kappa}\mathbb{E}_{\mathbf{x}}^{\mathbb{Q}^{{\bm \theta}}}\left[e^{-r\tau^\ast}\frac{\hat{F}(\mathbf{a}^T\mathbf{X}_{\tau^\ast})}{U_0(\mathbf{a}^T\mathbf{X}_{\tau^\ast})}
U_0(\mathbf{a}^T\mathbf{X}_{\tau^\ast})\mathbbm{1}_{\{\tau^\ast<\infty\}}\right]\\
&\geq\left(\frac{\hat{F}(-x^\ast)}{U_0(-x^\ast)}\wedge\frac{\hat{F}(x^\ast)}{U_0(x^\ast)}\right)\inf_{\mathbb{Q}^{{\bm \theta}}\in \mathcal{P}^\kappa}\mathbb{E}_{\mathbf{x}}^{\mathbb{Q}^{{\bm \theta}}}\left[e^{-r\tau^\ast}U_0(\mathbf{a}^T\mathbf{X}_{\tau^\ast})\mathbbm{1}_{\{\tau^\ast<\infty\}}\right]\\
&=\frac{\hat{F}(x^\ast)}{U_0(x^\ast)}\mathbb{E}_{\mathbf{x}}^{\mathbb{Q}^{{\bm \theta}^\ast}}\left[e^{-r\tau^\ast}U_0(\mathbf{a}^T\mathbf{X}_{\tau^\ast})\mathbbm{1}_{\{\tau^\ast<\infty\}}\right]=\frac{\hat{F}(x^\ast)}{U_0(x^\ast)}U_0(\mathbf{a}^T\mathbf{x})=\hat{V}_\kappa(\mathbf{x})
\end{align*}
proving that $\hat{V}_\kappa(\mathbf{x})=V_\kappa(\mathbf{x})$ for all  $\mathbf{x}\in\{\mathbf{x}\in\mathbb{R}^d:\mathbf{a}^T\mathbf{x}\in (-x^\ast,x^\ast)\}$ and that $V_\kappa(\mathbf{x})=\hat{F}(\mathbf{a}^T\mathbf{x})$ for $\mathbf{x}\in\{\mathbf{x}\in\mathbb{R}^d:\mathbf{a}^T\mathbf{x}=-x^\ast \textrm{or }\mathbf{a}^T\mathbf{x}=x^\ast\}$. Finally, if $\mathbf{x}\in\{\mathbf{x}\in\mathbb{R}^d:\mathbf{a}^T\mathbf{x}\not\in (-x^\ast,x^\ast)\}$, then $\tau^\ast=0$ $\mathbb{Q}^{{\bm \theta}}$-almost surely and
\begin{align*}
V_\kappa(\mathbf{x})&\geq \inf_{\mathbb{Q}^{{\bm \theta}}\in \mathcal{P}^\kappa}\mathbb{E}_{\mathbf{x}}^{\mathbb{Q}^{{\bm \theta}}}\left[e^{-r\tau^\ast}\hat{F}(\mathbf{a}^T\mathbf{X}_{\tau^\ast})\mathbbm{1}_{\{\tau^\ast<\infty\}}\right]\\
&= \hat{F}(\mathbf{a}^T\mathbf{x}) =\hat{V}_\kappa(\mathbf{x})
\end{align*}
completing the proof of our theorem.
\end{proof}

\begin{rmk}
It is worth pointing out that the positive homogeneity of degree $-1$ of the constants $\psi_\kappa,\varphi_\kappa,\hat{\psi}_\kappa,\hat{\varphi}_\kappa$ as functions of the parameter vector $\mathbf{a}$ guarantees that the function $U_c$ remains unchanged for parameter vectors of equal Euclidean length, that is, for vectors satisfying the condition $\|\mathbf{a}_1\|=\|\mathbf{a}_2\|$. Consequently, solving the stopping problem with respect one $\mathbf{a}_1 \in \mathbb{R}^d$ results into an optimal policy and value for an entire class of problems constrained by the requirement that $\|\mathbf{a}_1\|=\|\mathbf{a}_2\|$.\\ It is furthermore interesting to note that already in dimension $d=1$ the underlying process under the worst case measure is a Brownian motion with \emph{broken drift} as studied in \cite{ernesto_paavo_broken_drift}. Therefore, in the class of problems studied in this paper, optimal stopping problems with broken drift naturally arise. Here, however, the breaking point always lies in the continuation set.
\end{rmk}
Theorem \ref{even} characterizes the optimal timing policy in the symmetric case where the exercise payoff is even and the ratio $\hat{F}(y)/U_0(y)$ attains a unique global maximum on $\mathbb{R}_+$ (and by symmetry also on $\mathbb{R}_-$). The findings of Theorem \ref{even} clearly indicate that in the present setting symmetry is useful in the characterization of the value and the worst case measure. To see that this is indeed the case, we now present a general observation valid for symmetric periodic payoffs.
\begin{theorem}\label{periodicpayoff}
Assume that the exercise payoff $\hat{F}(x)$ satisfies the following conditions
\begin{itemize}
  \item[(A)] The function $\hat{F}(x)$ is periodic with period length $P>0$;
  \item[(B)] There exists a threshold $x_1\in \mathbb{R}$ so that $\hat{F}(x_1)\geq \hat{F}(x) \geq \hat{F}(x_0)$, where $x_0=x_1-P/2$, for all $x\in \mathbb{R}$;
  \item[(C)] The function $\hat{F}(x)$ satisfies the symmetry condition $\hat{F}(x_0-x)=\hat{F}(x_0+x)$ for all $x\in[0,P/2]$.
\end{itemize}
Assume also that there exists a unique interior threshold
\begin{align*}
x^\ast =\argmax_{x\in[x_0,x_1]}\left\{\frac{\hat{F}(x)}{U_{x_0}(x)}\right\},
\end{align*}
so that $\hat{F}(x)/U_{x_0}(x)$ is increasing on $(x_0,x^\ast)$ and decreasing on $(x^\ast,x_1)$. Then, the value of the optimal stopping policy $\inf\{t\geq 0:\mathbf{a}^T\mathbf{X}_t\not\in \cup_{n\in \mathbb{Z}}(y_n^\ast,z_n^\ast)\}$ reads as
\begin{align}\label{value}
V_\kappa(\mathbf{x})=\begin{cases}
\hat{F}(\mathbf{a}^T\mathbf{x}),&\mathbf{a}^T\mathbf{x}\not\in \cup_{n\in \mathbb{Z}}(y_n^\ast,z_n^\ast),\\
\frac{\hat{F}(x^\ast)}{U_{x_0}(x^\ast)}U_{x_0}(\mathbf{a}^T\mathbf{x}),&\mathbf{a}^T\mathbf{x}\in \cup_{n\in \mathbb{Z}}(y_n^\ast,z_n^\ast),
\end{cases}
\end{align}
where $y_n^\ast=2x_0-x^\ast + nP$ and $z_n^\ast=x^\ast + n P$, $n\in \mathbb{Z}$.
Moreover, the optimal density generator resulting into the worst case measure is
$$
{\bm\theta}_t^\ast=\begin{cases}
-\kappa\frac{{\bm a}}{\|{\bm a}\|}, & \mathbf{a}^T\mathbf{X}_t\in\cup_{n\in \mathbb{Z}}[x_0+nP,x_1+nP]\\
\kappa\frac{{\bm a}}{\|{\bm a}\|}, & \mathbf{a}^T\mathbf{X}_t\in\cup_{n\in \mathbb{Z}}[x_1+nP,x_0+(n+1)P]
\end{cases}
$$
\end{theorem}
\begin{proof}
The assumed periodicity and symmetry of the exercise payoff $\hat{F}$ implies that we can focus on the behavior of the ratio
$\hat{F}(y)/U_{x_0}(y)$
on $[x_1-P,x_1]$ (from a maximum to the next). It is clear that since $\hat{\psi}_{\kappa}=-\varphi_\kappa$ and $\hat{\varphi}_{\kappa}=-\psi_\kappa$ we have
\begin{align*}
U_{x_0}(x_0-x)&= \frac{\hat{\psi}_{\kappa}}{\hat{\psi}_{\kappa}-\hat{\varphi}_{\kappa}}e^{-\hat{\varphi}_{\kappa} x}-\frac{\hat{\varphi}_{\kappa}}{\hat{\psi}_{\kappa}-\hat{\varphi}_{\kappa}}e^{-\hat{\psi}_{\kappa} x}\\
&= \frac{\psi_{\kappa}}{\psi_{\kappa}-\varphi_{\kappa}}e^{\varphi_{\kappa} x}-\frac{\varphi_{\kappa}}{\psi_{\kappa}-\varphi_{\kappa}}e^{\psi_{\kappa} x} = U_{x_0}(x_0+x)
\end{align*}
for $x\in[0,P/2]$.
Consequently, assumption (C) guarantees that
$$
\frac{\hat{F}(x_0+x)}{U_{x_0}(x_0+x)}=\frac{\hat{F}(x_0-x)}{U_{x_0}(x_0-x)}
$$
for all $x\in[0,P/2]$. On the other hand, our assumption on the existence of an interior maximizing threshold $x^\ast$ and the symmetry of $\hat{F}$ guarantees that
$$
2x_0-x^\ast=\argmax_{x\in[x_1-P,x_0]}\left\{\frac{\hat{F}(x)}{U_{x_0}(x)}\right\}
$$
and
$$
\frac{\hat{F}(x^\ast)}{U_{x_0}(x^\ast)}=\frac{\hat{F}(2x_0-x^\ast)}{U_{x_0}(2x_0-x^\ast)}.
$$
Combining this result with the assumed periodicity of the payoff then shows that
\begin{align*}
z_n^\ast&=x^\ast + n P = \argmax_{x\in[x_0+nP,x_1+nP]}\left\{\frac{\hat{F}(x)}{U_{x_0+nP}(x)}\right\}\\
y_n^\ast&=2x_0-x^\ast + n P = \argmax_{x\in[x_1+(n-1)P,x_0+nP]}\left\{\frac{\hat{F}(x)}{U_{x_0+nP}(x)}\right\}.
\end{align*}
The alleged optimality and characterization of the optimal density generator is now identical with the proof of our Theorem \ref{even}.
\end{proof}

\subsection{Discontinuous Asymmetric Digital Option}

In order to illustrate our general findings, we now focus on the discontinuous asymmetric digital option case, where $\hat{F}(x)=(k_2 x + k_3)\mathbbm{1}_{\{x\geq0\}}-k_1 x\mathbbm{1}_{\{x<0\}}$, where $k_1,k_2,k_3\in \mathbb{R}_+$ are known positive constants. In the present setting it suffices to investigate the behavior of the functions
$
\Pi_1(x)=(k_2 x + k_3)/h_{1c}(x)
$
and
$
\Pi_2(x)=-k_1 x/h_{2c}(x).
$
Standard differentiation yields $\Pi_1'(x)=f_1(x)/h_{1c}^2(x)$ and $\Pi_2'(x)=k_1f_2(x)/h_{2c}^2(x)$, where
\begin{align*}
f_1(x) &=k_2 h_{1c}(x) - h_{1c}'(x)(k_2 x + k_3)\\
f_2(x) &= h_{2c}'(x)x- h_{2c}(x).
\end{align*}
Since $f_1(c)=k_2>0,f_2(c)=-1<0$
\begin{align*}
f_1'(x) &= - h_{1c}''(x)(k_2 x + k_3)\\
f_2'(x) &= h_{2c}''(x)x
\end{align*}
$\lim_{x\rightarrow\infty}f_1(x)=-\infty$, and $\lim_{x\rightarrow-\infty}f_2(x)=\infty$ we notice that there exists two thresholds $x_1^\ast(c)>c\vee -k_3/k_2$ and $x_2^\ast(c)<c\wedge 0$ so that the first order conditions
$f_1(x_1^\ast(c))=0, f_2(x_2^\ast(c))=0$ are satisfied. Moreover, the thresholds $x_1^\ast(c), x_2^\ast(c)$ are increasing as functions of the reference point $c$ and satisfy the limiting conditions $\lim_{c\rightarrow-\infty}x_1^\ast(c)=-k_3/k_2+1/\psi_\kappa,\lim_{c\rightarrow-\infty}x_2^\ast(c)=-\infty$, $\lim_{c\rightarrow\infty}x_1^\ast(c)=\infty,$ and $\lim_{c\rightarrow\infty}x_2^\ast(c)=1/\hat{\varphi}_\kappa$.
Thus, we notice by utilizing our results above that $\lim_{c\rightarrow-\infty}\Pi_1(x_1^\ast(c))=0,\lim_{c\rightarrow\infty}\Pi_1(x_1^\ast(c))=\infty,$
$\lim_{c\rightarrow-\infty}\Pi_2(x_2^\ast(c))=\infty$, and $\lim_{c\rightarrow\infty}\Pi_2(x_2^\ast(c))=0$. Consequently, we notice that there is a unique $\hat{c}$ such that $\Pi_1(x_1^\ast(\hat{c}))=\Pi_2(x_2^\ast(\hat{c}))$ is met. Two cases arise. If $x_1^\ast(\hat{c})\geq 0$, then $c^\ast=\hat{c}$ is the optimal state at which the density generator switches from one extreme to another and the value of the optimal policy reads as
$$
V_\kappa(\mathbf{x})=\begin{cases}
k_2 \mathbf{a}^T\mathbf{x} + k_3,&\mathbf{a}^T\mathbf{x}\geq x_1^\ast(c^\ast),\\
\Pi_1(x_1^\ast(c^\ast))U_{c^\ast}(\mathbf{a}^T\mathbf{x}),&x_2^\ast(c^\ast) < \mathbf{a}^T\mathbf{x}< x_1^\ast(c^\ast),\\
-k_1 \mathbf{a}^T\mathbf{x},&\mathbf{a}^T\mathbf{x}\leq x_2^\ast(c^\ast).
\end{cases}
$$
Especially, the value satisfies the smooth-fit condition at the optimal boundaries $x_1^\ast(c^\ast)$ and $x_2^\ast(c^\ast)$. This case is illustrated in Figure \ref{multiBM} under the assumptions that $\|\mathbf{a}\| = 0.1, r = 0.02, k_1=1, k_2 = 0.5$, and $k_3 = 0.35$ (implying that $c^\ast=-0.0941818, x_2^\ast=-0.616587$, and $x_1^\ast=0.205943$)
\begin{figure}[!ht]
\begin{center}
\includegraphics[width=0.6\textwidth]{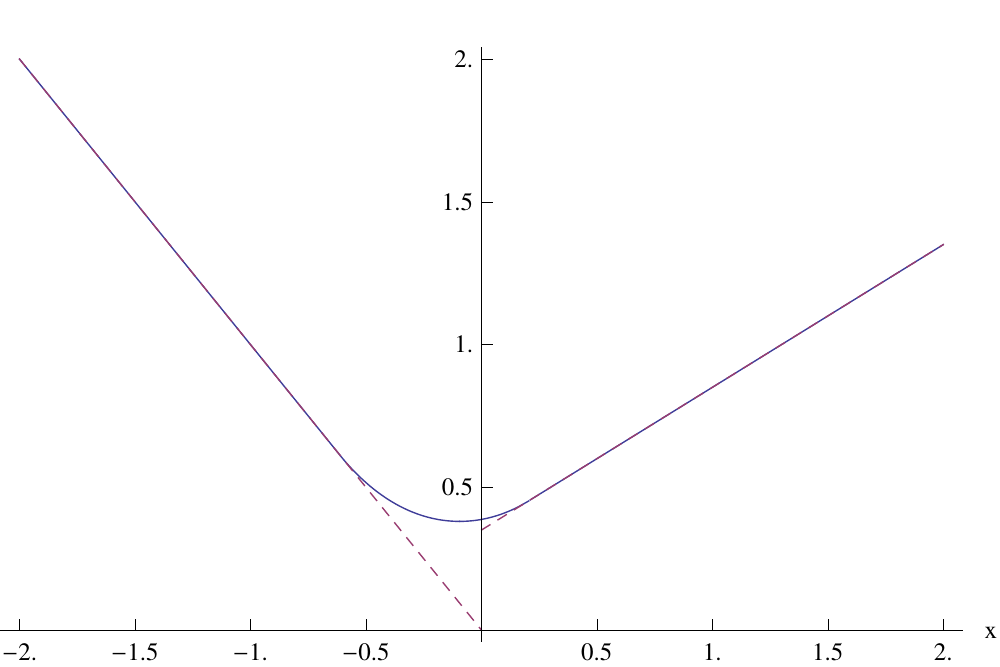}
\caption{\small The value function (uniform) and exercise payoff (dashed)}\label{multiBM}
\end{center}
\end{figure}

However, if $x_1^\ast(\hat{c})< 0$ then the situation changes since in that case $0$ becomes an optimal stopping boundary at which the value coincides with the payoff in a nondifferentiable way. In that case the value reads as
$$
V_\kappa(\mathbf{x})=\begin{cases}
k_2 \mathbf{a}^T\mathbf{x} + k_3,&\mathbf{a}^T\mathbf{x}\geq 0,\\
\Pi_2(x_2^\ast(c^\ast))U_{c^\ast}(\mathbf{a}^T\mathbf{x}),&x_2^\ast(c^\ast) < \mathbf{a}^T\mathbf{x}< 0,\\
-k_1 \mathbf{a}^T\mathbf{x},&\mathbf{a}^T\mathbf{x}\leq x_2^\ast(c^\ast),
\end{cases}
$$
where the optimal boundary and the critical switching state are the unique roots of the equations
\begin{align*}
h_{2c^{\ast}}'(x_2^\ast(c^\ast))x_2^\ast(c^\ast)&=h_{2c^\ast}(x_2^\ast(c^\ast))\\
-\frac{k_1 x_2^\ast(c^\ast)}{h_{2c^\ast}(x_2^\ast(c^\ast))} &= \frac{k_3}{h_{1c^\ast}(0)}.
\end{align*}
This case is illustrated in Figure \ref{multiBM2} under the assumptions that $\|\mathbf{a}\| = 0.1, r = 0.02, k_1=1, k_2 = 0.5$, and $k_3 = 0.7$ (implying that $c^\ast=-0.348597, x_2^\ast=-0.739769$, and $x_1^\ast=0$)
\begin{figure}[!ht]
\begin{center}
\includegraphics[width=0.6\textwidth]{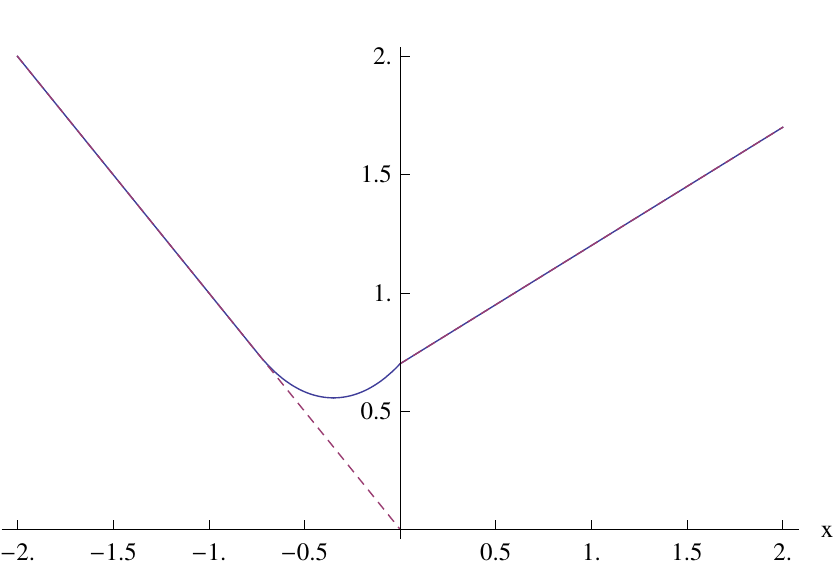}
\caption{\small The value function (uniform) and exercise payoff (dashed)}\label{multiBM2}
\end{center}
\end{figure}

It is at this point worth pointing out that in the case where $k_3=0$ and $k_1=k_2$ the exercise payoff is continuous and even and the findings of Theorem \ref{even} applies.
In that case, the optimal boundaries can be solved from the optimality condition $\psi_\kappa  e^{\psi_\kappa x^\ast}(1-\varphi_\kappa x^\ast)=\varphi_\kappa e^{\varphi_\kappa x^\ast}(1-\psi_\kappa x^\ast)$.

\subsection{Periodic and Even Payoff}

In order to illustrate how the approach applies in the periodic setting resulting into multiple boundaries, consider the periodic payoff $\hat{F}(x)=\cos(x)$. Since the payoff is even, attains its maxima at the points $y_n = 2n\pi$, its minima at the points $x_n = (2n+1)\pi$, and is symmetric on the sets $[2n\pi,2(n+1)\pi]$, $n\in \mathbb{Z}$, we notice that we can extend the findings of Theorem \ref{even} and make an ansatz that the optimal reference point is $c_n^\ast=x_n$. To see that this is indeed the case, we first observe that if $y\in [y_n,x_n]$ then  $\Pi_{x_n}(x_n+y)=\Pi_{x_n}(x_n-y)$, since
$\cos(x_n-y)=\cos(x_n+y)$ and
\begin{align*}
U_{x_n}(x_n-y)&=\frac{\hat{\psi}_{\kappa}}{\hat{\psi}_{\kappa}-\hat{\varphi}_{\kappa}}e^{-\hat{\varphi}_{\kappa} y}-\frac{\hat{\varphi}_{\kappa}}{\hat{\psi}_{\kappa}-\hat{\varphi}_{\kappa}}e^{-\hat{\psi}_{\kappa} y}\\
&=-\frac{\varphi_{\kappa}}{\psi_{\kappa}-\varphi_{\kappa}}e^{\psi_{\kappa} y}+\frac{\psi_{\kappa}}{\psi_{\kappa}-\varphi_{\kappa}}e^{\varphi_{\kappa} y}=U_{x_n}(x_n+y)
\end{align*}
for all $y\in \mathbb{R}$ and $n\in \mathbb{Z}$. Consequently, it is sufficient to investigate the ratio
$\Pi_{x_n}(y)$ on $[x_n,y_{n+1}]$. Standard differentiation yields
$\Pi_{x_n}'(y)=u_n(y)/U_{x_n}^2(y)$, where
$$
u_n(y) = \frac{\varphi_{\kappa}}{\psi_{\kappa}-\varphi_{\kappa}}e^{\psi_{\kappa} (y-x_n)}(\sin(y)+\psi_\kappa \cos(y))-\frac{\psi_{\kappa}}{\psi_{\kappa}-\varphi_{\kappa}}e^{\varphi_{\kappa} (y-x_n)}(\sin(y)+\varphi_\kappa \cos(y)).
$$
Noticing now that $u_n(x_n)=0$, $$u_n(y_{n+1})=\frac{\psi_{\kappa}\varphi_{\kappa}}{\psi_{\kappa}-\varphi_{\kappa}}\left(e^{\psi_{\kappa} \pi}-e^{\varphi_{\kappa} \pi}\right)<0,$$ and
$$
u_n'(y) = \left(\varphi_{\kappa}(\psi_{\kappa}^2+1)e^{\psi_{\kappa} (y-x_n)}-\psi_{\kappa}(\varphi_{\kappa}^2+1)e^{\varphi_{\kappa} (y-x_n)}\right)\frac{\cos(y)}{\psi_{\kappa}-\varphi_{\kappa}}
$$
we notice that equation $u_n(y)=0$ has a unique root $z_n^\ast\in(x_n+\frac{\pi}{2},y_{n+1})$ such that
$$
z_n^\ast=\argmax_{y\in[x_n,y_{n+1}]}\Pi_{x_n}(y).
$$
It is now clear that the value of the optimal stopping policy reads as
$$
V_\kappa(\mathbf{x})=\begin{cases}
\Pi_{x_n}(z_n^\ast) U_{x_n}(\mathbf{a}^T\mathbf{x}),& \mathbf{a}^T\mathbf{x}\in \cup_{n\in \mathbb{Z}}(x_n-z_n^\ast,x_n+z_n^\ast),\\
\cos(\mathbf{a}^T\mathbf{x}),&\mathbf{a}^T\mathbf{x}\not\in \cup_{n\in \mathbb{Z}}(x_n-z_n^\ast,x_n+z_n^\ast).
\end{cases}
$$
This value and the optimal policies are illustrated for $y\in[-2\pi,2\pi]$ in Figure \ref{periodic} under the assumptions that $\kappa = 0.02, r = 0.03,$ and  $\sigma = 0.1$ (implying that the optimal thresholds are $-5.07233,-1.21086,1.21086,5.07233$).
\begin{figure}[!ht]
\begin{center}
\includegraphics[width=0.6\textwidth]{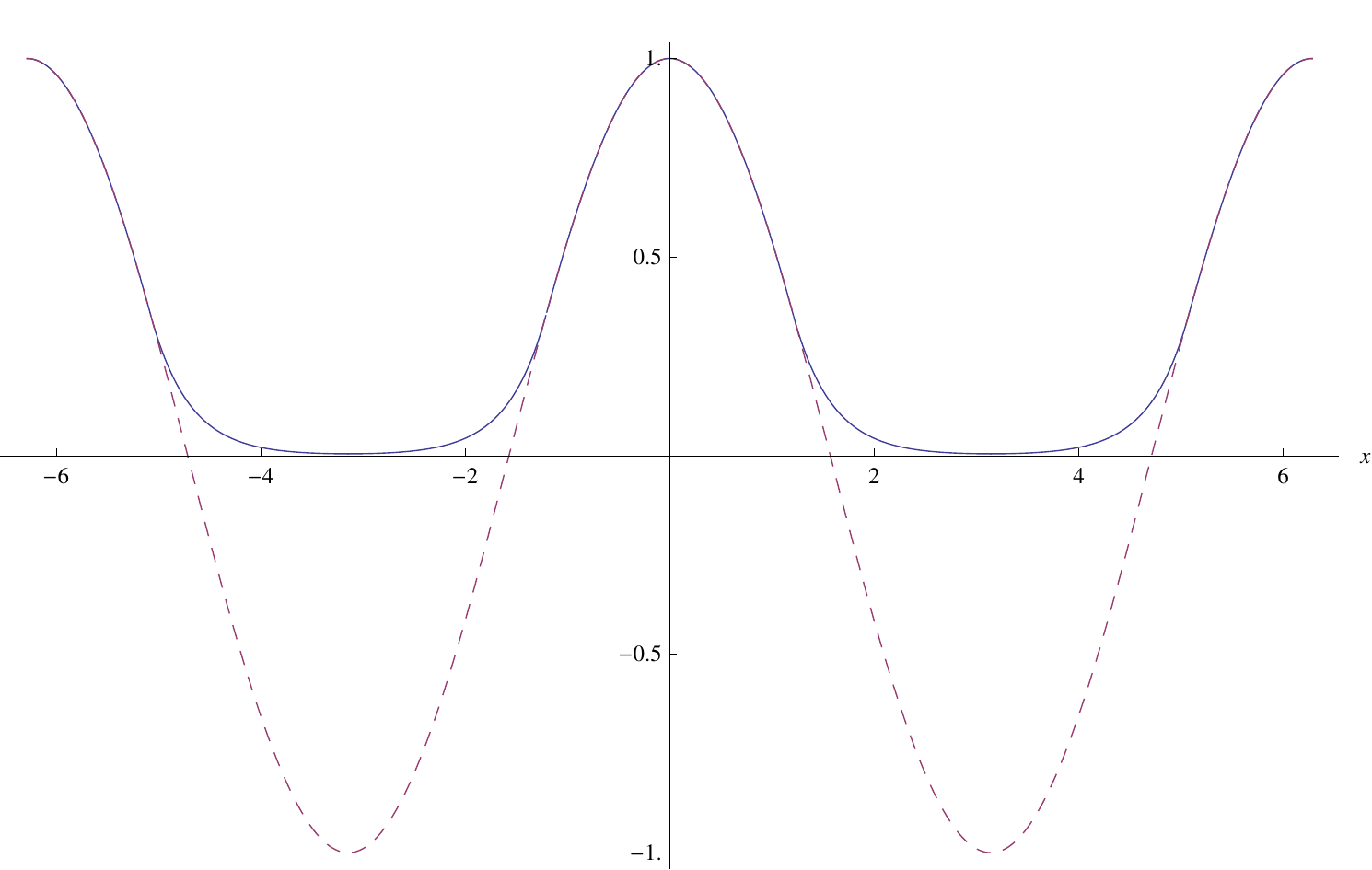}
\caption{\small The value function (uniform) and exercise payoff (dashed) in the periodic case}\label{periodic}
\end{center}
\end{figure}
It is worth noticing that the worst case prior is induced in the present case by the density generator
$$
{\bm\theta}^\ast=\begin{cases}
\kappa\frac{{\bm a}}{\|{\bm a}\|}, &(2n+1)\pi \leq \mathbf{a}^T\mathbf{x} \leq 2(n+1)\pi,\\
-\kappa\frac{{\bm a}}{\|{\bm a}\|}, &2n\pi \leq \mathbf{a}^T\mathbf{x} \leq (2n+1)\pi,
\end{cases}
$$
for all $n\in \mathbb{Z}$. Essentially, the optimal density generator tends to drive the dynamics of the underlying diffusion towards the
minim points $x_n$ of the exercise payoff.\\

\section{Radially Symmetric Payoff}

It is well-known from the literature on linear diffusions that the radial part of a multidimensional Brownian motion constitutes a Bessel process. Our objective is now to exploit this connection by focusing on exercise payoffs which are radially symmetric. More precisely, we now assume that the payoff is of the form
\begin{align}\label{radialpayoff}
F({\bf x})=\hat{F}\left(\|\bm x\|^2\right) = \hat{F}\left(\sum_{i=1}^dx_i^2\right),
\end{align}
where $\hat{F}:\mathbb{R}_+\mapsto \mathbb{R}$ is a known measurable function.
We now make an ansatz and focus on functions which are radially symmetric, that is, on functions of the form
\begin{align*}
u({\bf x})= h\left(\|\bm x\|^2\right) = h\left(\sum_{i=1}^dx_i^2\right),
\end{align*}
where $h:\mathbb{R}_+\mapsto \mathbb{R}_+$ is assumed to be twice continuously differentiable on $\mathbb{R}_+$. In this case, a short calculation yields that the worst case prior becomes
\[{\bm\theta}^*=\kappa\sgn(h'(\|{\bm x}\|^2)) \frac{{\bm x}}{\|{\bm x}\|},\]
so that the worst case drift points towards the origin or away from it, resp. In this case, solving
\[(\mathcal{A}^{{\bm \theta}^\ast}u)({\bm x})=ru({\bm x})\]
results into solving
\[2(\|{\bm x}\|^2)h''(\|{\bm x}\|^2)+\left(d-2\kappa\|{\bm x}\|\right)h'(\|{\bm x}\|^2)=rh(\|{\bm x}\|^2)\]
on $\{{\bm x}\in \mathbb{R}^d:h'(\|{\bm x}\|^2)\geq 0\}$ and
\[2(\|{\bm x}\|^2)h''(\|{\bm x}\|^2)+\left(d+2\kappa\|{\bm x}\|\right)h'(\|{\bm x}\|^2)=rh(\|{\bm x}\|^2)\]
on on $\{{\bm x}\in \mathbb{R}^d:h'(\|{\bm x}\|^2)\leq0\}$.  Denote now by $M_{a,b}$ and by $W_{a,b}$ the {\em Whittaker functions} of the first and second type, respectively, and define
the functions $\psi_1(y)=u_{\kappa}(\sqrt{y}), \varphi_1(y)=v_{\kappa}(\sqrt{y}), \psi_2(y)=u_{-\kappa}(\sqrt{y})$, and  $\varphi_2(y)=v_{-\kappa}(\sqrt{y})$, where
\begin{align*}
u_{\kappa}(y)&= y^{\frac{1-d}{2}}e^{\kappa y}M_{a_\kappa,b}\left(2\sqrt{2r+\kappa^2}y\right)\\
v_{\kappa}(y)&= y^{\frac{1-d}{2}}e^{\kappa y} W_{a_\kappa,b}\left(2\sqrt{2r+\kappa^2}y\right),
\end{align*}
$b=d/2-1$, and
$$
a_\kappa=\frac{\kappa(d-1)}{2\sqrt{\kappa^2+2r}}.
$$
Making the substitution $h(\|{\bm x}\|^2)=v(\|{\bm x}\|)$ show that the solutions of these ODEs read as (cf. \cite{Lin2004})
\begin{align*}
h\left(\|\bm x\|^2\right) = c_1\psi_1(\|\bm x\|^2)+c_2\varphi_1(\|\bm x\|^2)
\end{align*}
on $\{{\bm x}\in \mathbb{R}^d:h'(\|{\bm x}\|^2)\geq 0\}$ and as
\begin{align*}
h\left(\|\bm x\|^2\right) =  \hat{c}_1\psi_2(\|\bm x\|^2)+\hat{c}_2\varphi_2(\|\bm x\|^2)
\end{align*}
on $\{{\bm x}\in \mathbb{R}^d:h'(\|{\bm x}\|^2)\leq 0\}$. As in the previous subsection, we now let $c\in \mathbb{R}_+$ be an arbitrary reference point and define the twice continuously differentiable function $U_c$ as the solution of the boundary value problem
\begin{align}\label{minimal}
\begin{split}
&2(\|{\bm x}\|^2)U_c''(\|{\bm x}\|^2)+\left(d-2\kappa\|{\bm x}\|\sgn(\|{\bm x}\|^2-c)\right)U_c'(\|{\bm x}\|^2)-rU_c(\|{\bm x}\|^2)=0\\
&U_c(c)=1,\quad U_c'(c)=0.
\end{split}
\end{align}
We again find that $U_c(\|\mathbf{x}\|^2)=\max(\hat{h}_{1c}(\|\mathbf{x}\|^2),\hat{h}_{2c}(\|\mathbf{x}\|^2))$, where
\begin{align*}
\hat{h}_{1c}(\|\mathbf{x}\|^2)&=B_1^{-1}\left(\frac{\psi_1'(c)}{S_1'(c)}\varphi_1(\|\mathbf{x}\|^2)-\frac{\varphi_1'(c)}{S_1'(c)}\psi_1(\|\mathbf{x}\|^2)\right),\\
\hat{h}_{2c}(\|\mathbf{x}\|^2)&=B_2^{-1}\left(\frac{\psi_2'(c)}{S_2'(c)}\varphi_2(\|\mathbf{x}\|^2)-\frac{\varphi_2'(c)}{S_2'(c)}\psi_2(\|\mathbf{x}\|^2)\right),\\
B_{1}&=\frac{\sqrt{2r+\kappa^2}\Gamma(d-1)}{\Gamma\left(\frac{d-1}{2}-a_\kappa\right)},\\
B_{2}&=\frac{\sqrt{2r+\kappa^2}\Gamma(d-1)}{\Gamma\left(\frac{d-1}{2}-a_{-\kappa}\right)},
\end{align*}
$S_1'(y)=e^{2\kappa\sqrt{y}}y^{-\frac{d}{2}}$, and $S_2'(y)=e^{-2\kappa\sqrt{y}}y^{-\frac{d}{2}}$. As in the case of the previous subsection, we define the two  cases associated with the extreme reference points by
\begin{align*}
U_0(y)=\psi_1(y) &=(2\sqrt{\kappa^2+2r})^{\frac{d-1}{2}} e^{\left(\kappa-\sqrt{\kappa^2+2 r}\right)\sqrt{y}} \tilde{M}\left(\frac{(d-1)}{2} \left(1-\frac{\kappa}{\sqrt{\kappa^2+2 r}}\right),d-1,2 \sqrt{\kappa^2+2 r} \sqrt{y}\right)\\
U_\infty(y)= \varphi_2(y) &=(2\sqrt{\kappa^2+2r})^{\frac{d-1}{2}}  e^{-\left(\sqrt{\kappa^2+2 r}+\kappa\right)\sqrt{y}} \tilde{U}\left(\frac{(d-1)}{2}  \left(1+\frac{\kappa}{\sqrt{\kappa^2+2 r}}\right),d-1,2 \sqrt{\kappa^2+2 r}
   \sqrt{y}\right)
\end{align*}
where $\tilde{M}$ and $\tilde{U}$ denote the confluent hypergeometric functions of the first and second type, respectively. It is worth noticing that since the lower boundary is entrance for the underlying diffusion process we have that (cf. p. 19 in \cite{BS15})
\begin{align*}
\lim_{y\rightarrow 0+}\hat{h}_{ic}(y)&=\infty\\
\lim_{y\rightarrow 0+}\frac{\hat{h}_{ic}'(y)}{S_i'(y)}&=B_i^{-1}\frac{\psi_i'(c)}{S_i'(c)}\lim_{y\rightarrow 0+}\frac{\varphi_i'(y)}{S_i'(y)}>-\infty
\end{align*}
when $c\in(0,\infty)$. The upper boundary is, in turn, natural for the underlying diffusion process and, hence, we have that (cf. p. 19 in \cite{BS15})
\begin{align*}
\lim_{y\rightarrow \infty}\hat{h}_{ic}(y)&=+\infty\\
\lim_{y\rightarrow \infty}\frac{\hat{h}_{ic}'(y)}{S_i'(y)}&=+\infty
\end{align*}
when $c\in(0,\infty)$. However, in contrast with natural boundary behavior, we now notice that in the extreme case
$$
\lim_{y\rightarrow0+}\psi_1(y)= (2\sqrt{\kappa^2+2r})^{\frac{d-1}{2}}.
$$
Again, we observe that the function $U_c$ is convex.
\begin{lemma}
The function $U_c(y)$ is strictly convex on $\mathbb{R}_+$.
\end{lemma}
\begin{proof}
$U_c(y)$ is nonnegative and decreasing on $(0,c]$. Consequently, we notice by invoking \eqref{minimal} that
$$
2yU_c''(y)=rU_c(y)-\left(d+2\kappa\sqrt{y}\right)U_c'(y)>0
$$
demonstrating that $U_c(y)$ is strictly convex on $(0,c]$. On the other hand,  \eqref{minimal} also implies that on $(c,\infty)$ we have
$$
\frac{2yU_c''(y)}{S_1'(y)}=\frac{r(U_c(y)-yU_c'(y))}{S_1'(y)}-\left(d-2\kappa\sqrt{y}-ry\right)\frac{U_c'(y)}{S_1'(y)}.
$$
Since
$$
\frac{d}{dy}\frac{U_c(y)-yU_c'(y)}{S_1'(y)}=\left(d-2\kappa\sqrt{y}-ry\right)U_c(y)m_1'(y),
$$
where $m_1'(y)=1/(2yS_1'(y))$, we notice by integrating from $c$ to $y$ that
$$
r\frac{U_c(y)-yU_c'(y)}{S_1'(y)} = \frac{r}{S_1'(c)}+r\int_c^y \left(d-2\kappa\sqrt{t}-rt\right)U_c(t)m_1'(t)dt.
$$
On the other hand, since
$$
\frac{U_c'(y)}{S_1'(y)} = r\int_c^yU_c(t)m_1'(t)dt
$$
we finally find that
$$
\frac{2yU_c''(y)}{S_1'(y)}=\frac{r}{S_1'(c)}+r\int_c^y \left(2\kappa(\sqrt{y}-\sqrt{t})+r(y-t)\right)U_c(t)m_1'(t)dt >0
$$
proving that $U_c(y)$ is strictly convex on $(c,\infty)$ as well.
\end{proof}

Utilizing the It{\^o}-Döblin theorem now shows that the process $Y_t=\|\mathbf{X}_t\|^2$ satisfies the SDE
\begin{align}\label{sde2}
dY_t = \left(d-2\kappa\sqrt{Y_t}\sgn(Y_t-c)\right)dt+2\sqrt{Y_t}d\tilde{W}_t^{{\bm \theta}_c},\quad Y_0=\|\mathbf{x}\|^2,
\end{align}
where $\tilde{W}_t^{{\bm \theta}_c}$ is a Brownian motion under the measure $\mathbb{Q}^{{\bm \theta}_c}$ characterized by the density generator
$$
{\bm \theta}_{ct}=\kappa \sgn(\|\mathbf{X}_t\|^2-c) \frac{\mathbf{X}_t}{\|\mathbf{X}_t\|}.
$$
Hence, we again observe that the controlled process has a stationary distribution for a fixed reference point $c$.
In the present case it reads as
$$
p_c(y)=\frac{m_c'(y)}{m_c(0,\infty)}
$$
where
$$
m_c'(y)=\frac{1}{2}y^{\frac{d}{2}-1}e^{-2\kappa|\sqrt{y}-\sqrt{c}|}
$$
and
$$
m_c(0,\infty)=\frac{1}{2}(2\kappa)^{-d}\left(e^{2\kappa\sqrt{c}}\Gamma(d,2\kappa \sqrt{c})+e^{-2\kappa\sqrt{c}}\int_0^{2\kappa\sqrt{c}}t^{d-1}e^{t}dt\right).
$$
It is also worth noticing that utilizing the It{\^o}-Döblin theorem to the process $Z_t:=\sqrt{Y_t}=\|\mathbf{X}_t\|$ results into the SDE
$$
dZ_t=\left(\frac{d-1}{2Z_t}-\kappa\sgn(Z_t-\sqrt{c})\right)dt+d\tilde{W}_t^{{\bm \theta}_c},\quad Z_0=\|\mathbf{x}\|,
$$
which constitutes a Bessel process of order $d/2-1$ with an alternating drift.

A modified characterization of the representation presented in Theorem \ref{thm1} is naturally valid in this case as well, since in the present case the set of admissible reference points is $[0,\infty]$. It is also worth noticing that the function $U_c(y)$ is no longer symmetric and, hence,  similar representations with the ones developed in Theorem \ref{even} and in Theorem \ref{periodic} are no longer possible. Moreover, since the lower boundary is entrance for the underlying process, policies which are radically different from the case considered in the previous section may appear. We will illustrate this point explicitly in the following subsection.

\subsection{Nonlinear Straddle Option}
In order to illustrate the peculiarities associated with the present case, let us consider the nonlinear straddle option case $\hat{F}(y)=|\sqrt{y}-K|$, where $K>0$ is an exogenously set fixed strike price. Consider first the behavior of the function
$$
(\mathcal{L}_{\psi_{1}}\hat{F})(y)=(\sqrt{y}-K)\frac{\psi_{1}'(y)}{S_1'(y)}-\frac{1}{2\sqrt{y}}\frac{\psi_{1}(y)}{S_1'(y)}.
$$
We notice that $(\mathcal{L}_{\psi_{1}}\hat{F})(0+)=0$ and
$$
(\mathcal{L}_{\psi_{1}}\hat{F})'(y)=\left(r(\sqrt{y}-K)+\kappa - \frac{d-1}{2\sqrt{y}}\right)\psi_{1}(y)m_1'(y)
$$
demonstrating that
$$
(\mathcal{L}_{\psi_{1}}\hat{F})(y)=\int_0^y\left(r(\sqrt{t}-K)+\kappa - \frac{d-1}{2\sqrt{t}}\right)\psi_{1}(t)m_1'(t)dt.
$$
Since $r(\sqrt{y}-K)+\kappa - (d-1)/(2\sqrt{y})$ is monotonically increasing and satisfies the inequality $r(\sqrt{y}-K)+\kappa - (d-1)/(2\sqrt{y})\gtreqqless 0$ for $y\gtreqqless \tilde{y}_0$, where $\tilde{y}_0$ is the unique root of $r(\sqrt{y}-K)+\kappa - (d-1)/(2\sqrt{y})=0$, we find that for $y>\hat{y}>\tilde{y}_0$ we have that
\begin{align*}
(\mathcal{L}_{\psi_{1}}\hat{F})(y)&=(\mathcal{L}_{\psi_{1}}\hat{F})(\hat{y})+\int_{\hat{y}}^y\left(r(\sqrt{t}-K)+\kappa - \frac{d-1}{2\sqrt{t}}\right)\psi_{1}(t)m_1'(t)dt\\
&\geq (\mathcal{L}_{\psi_{1}}\hat{F})(\hat{y})+\left(r(\sqrt{\hat{y}}-K)+\kappa - \frac{d-1}{2\sqrt{\hat{y}}}\right)\int_{\hat{y}}^y\psi_{1}(t)m_1'(t)dt\\
&= (\mathcal{L}_{\psi_{1}}\hat{F})(\hat{y})+\left(r(\sqrt{\hat{y}}-K)+\kappa - \frac{d-1}{2\sqrt{\hat{y}}}\right)\frac{1}{r}\left(\frac{\psi_{1}'(y)}{S_1'(y)}-\frac{\psi_{1}'(\hat{y})}{S_1'(\hat{y})}\right).
\end{align*}
Hence, $\lim_{y\rightarrow\infty}(\mathcal{L}_{\psi_{1}}\hat{F})(y)=\infty$ demonstrating that
there is a unique $y^{\ast}_{K}>\tilde{y}_0$ satisfying the condition
$
(\mathcal{L}_{\psi_{1}}\hat{F})(y^{\ast}_{K})=0.
$
Noticing that
$$
\frac{d}{dy}\frac{\sqrt{y}-K}{\psi_{1}(y)}=-\frac{S_1'(y)}{\psi_{1}^2(y)}(\mathcal{L}_{\psi_{1}}\hat{F})(y)
$$
in turn demonstrates that $y^{\ast}_{K}>K^2$ is the unique threshold at which the ratio
$$
\Pi_0(y)=\frac{\sqrt{y}-K}{\psi_{1}(y)}
$$
is maximized. Moreover, $\partial y^{\ast}_{K}/\partial K >0$, $\lim_{K\rightarrow\infty}y^{\ast}_{K}=\infty$, and $\lim_{K\rightarrow0+}y^{\ast}_{K}=y^{\ast}_{0}>0$, where the threshold $y^{\ast}_{0}\in \mathbb{R}_+$
is the unique root of the first order optimality condition
$$
\psi_1(y^\ast_{0})=2\psi_1'(y^\ast_{0})y^{\ast}_{0}.
$$
Define now the monotonically increasing and continuously differentiable function $\tilde{V}_\kappa:\mathbb{R}_+\mapsto \mathbb{R}_+$ as
$$
\tilde{V}_\kappa(y)= \psi_{1}(y)\sup_{x\geq y}\left\{\frac{\sqrt{x}-K}{\psi_{1}(x)}\right\}=\begin{cases}
\sqrt{y}-K,&y\in [y^{\ast}_{K},\infty),\\
\Pi_0(y^{\ast}_{K})\psi_{1}(y),&y\in(0,y^{\ast}_{K}).
\end{cases}
$$
Since $\tilde{V}_\kappa(y)$ is nonnegative and dominates $\sqrt{y}-K$ for all $y\in \mathbb{R}_+$ it dominates $(\sqrt{y}-K)^+$ as well. Hence, we observe by utilizing similar arguments as in the proof of Theorem \ref{even} that
$$
\tilde{V}_\kappa(y)=\sup_{\tau\in\mathcal{T}}\mathbb{E}_{y}^{\mathbb{Q}^{{\bm \theta}_0}}\left[e^{-r\tau}\left(\sqrt{Y_\tau}-K\right)^+\right].
$$
Given this function we immediately notice that if condition
$$
\lim_{y\rightarrow 0+}\Pi_0(y^{\ast}_{K})\psi_{1}(y) = \Pi_0(y^{\ast}_{K})(2\sqrt{\kappa^2+2r})^{\frac{d-1}{2}}\geq K
$$
is met, then $\tilde{V}_\kappa(y)$ dominates the exercise payoff $|\sqrt{y}-K|$ for all $y\in \mathbb{R}_+$ as well. Therefore, we notice that in that case
$$
\tilde{V}_\kappa(y)=\sup_{\tau\in\mathcal{T}}\mathbb{E}_{y}^{\mathbb{Q}^{{\bm \theta}_0}}\left[e^{-r\tau}|\sqrt{Y_\tau}-K|\right].
$$

However, if
\begin{align}\label{valatorig}
\lim_{y\rightarrow 0+}\Pi_0(y^{\ast}_{K})\psi_{1}(y) = \Pi_0(y^{\ast}_{K})(2\sqrt{\kappa^2+2r})^{\frac{d-1}{2}}<K
\end{align}
then the optimal policy is no longer a standard single boundary policy. To see that this is indeed the case
consider the behavior of the ratio
$$
\hat{\Pi}_c(y)=\frac{|\sqrt{y}-K|}{U_c(y)}
$$
for all $c\in (0,\infty)$ and $y\in \mathbb{R}_+$.  Define now for an arbitrary state $y\in \mathbb{R}_+$ the continuous difference $D:\mathbb{R}_+\mapsto \mathbb{R}$ as
$$
D(c)=\sup_{w\geq y}\hat{\Pi}_c(w)-\sup_{w\leq y}\hat{\Pi}_c(w).
$$
Consider first the extreme case $D(0)$. It is clear from our analysis on the single boundary case treated above that $\hat{\Pi}_0(y)$ is monotonically decreasing on $(0,K^2)\cup (y_K^\ast,\infty)$,
monotonically increasing on $(K^2,y_K^\ast)$, and satisfies the limiting conditions $\lim_{y\rightarrow \infty}\hat{\Pi}_0(y)=0$ and
$$
\lim_{y\rightarrow 0}\hat{\Pi}_0(y) = (2\sqrt{\kappa^2+2r})^{\frac{1-d}{2}}K>\Pi_0(y^{\ast}_{K})
$$
by assumption \eqref{valatorig}. Combining these observations show that $\hat{\Pi}_0(0)>\hat{\Pi}_0(y)$ for all $y\in \mathbb{R}_+$ and, consequently, that $D(0)<0$. Consider now, in turn, the other extreme setting $D(\infty)$. Utilizing now completely analogous arguments as before, we notice that $\hat{\Pi}_\infty(y)$ is monotonically increasing on $(K^2,\infty)$, bounded for $y\in(0,\infty)$,
and satisfies the limiting conditions $\hat{\Pi}_\infty(0)=0$ and
$
\lim_{y\rightarrow\infty}\hat{\Pi}_\infty(y)=\infty.
$
Consequently, we notice that for $y\in(0,\infty)$ we have $\sup_{w\leq y}\hat{\Pi}_\infty(w)<\infty$,
$$
\sup_{w\geq y}\hat{\Pi}_\infty(w)=\infty
$$
and, therefore, that $\lim_{c\rightarrow\infty}D(c)=\infty$. Combining these results with the continuity of the difference $D(c)$ proves that there is at least one $c^\ast\in\mathbb{R}_+$ such that
$D(c^\ast)=0$ implying that
$$
\sup_{w\geq y}\hat{\Pi}_{c^\ast}(w)=\sup_{w\leq y}\hat{\Pi}_{c^\ast}(w).
$$
Moreover, the optimal thresholds $y_i^\ast, i=1,2$ satisfy the ordinary first order optimality conditions
$$
\frac{h_{ic^\ast}(y_i^\ast)}{2\sqrt{y_i^\ast}}=h_{ic^\ast}'(y_i^\ast)(\sqrt{y_i^\ast}-K),\quad i=1,2.
$$
In this case the value reads as
$$
V_\kappa(y)=\begin{cases}
\sqrt{y}-K,&y\geq y_1^\ast\\
\hat{\Pi}_{c^\ast}(y_1^\ast)h_{1c^\ast}(y),&y_2^\ast<y<y_1^\ast\\
K-\sqrt{y},&y\leq y_2^\ast.
\end{cases}
$$
Naturally, the set $(0,y_2^\ast)\cup(y_1^\ast,\infty)$ constitutes the stopping set in the present example.

In order to illustrate our findings numerically, we now assume that $r = 0.1, \kappa = 0.02$, and $d = 5$ (implying that the critical cost below which the problem becomes a single boundary problem is $K\approx 0.975222$). The two boundary setting is illustrated in Figure \ref{Bessel} under the assumption that $K=4$ (implying that $y_2^\ast=3.85108, y_1^\ast=63.4344$, and $c^\ast=9.07278$).
\begin{figure}[!ht]
\begin{center}
\includegraphics[width=0.6\textwidth]{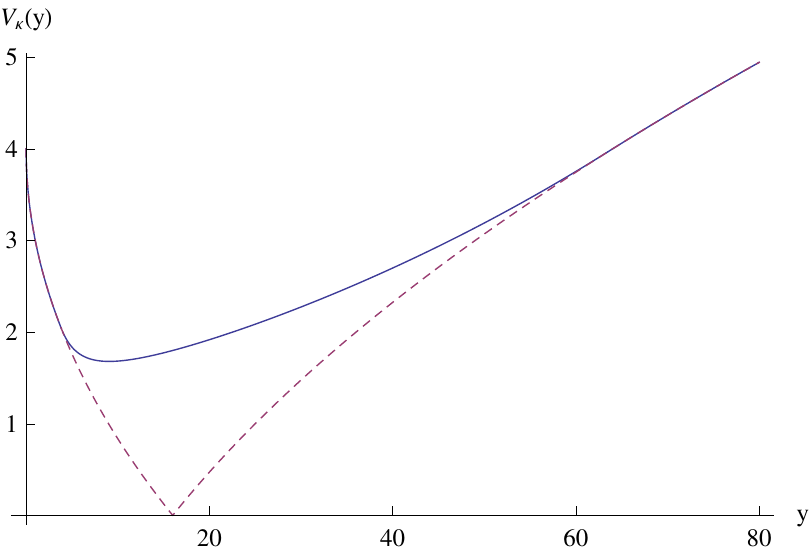}
\caption{\small The value (uniform) and exercise payoff (dashed)}\label{Bessel}
\end{center}
\end{figure}
The single boundary setting is, in turn, illustrated in Figure \ref{Bessel2} under the assumption that $K=0.85$ (implying that $y_{0.85}^\ast=4.7294$).
\begin{figure}[!ht]
\begin{center}
\includegraphics[width=0.6\textwidth]{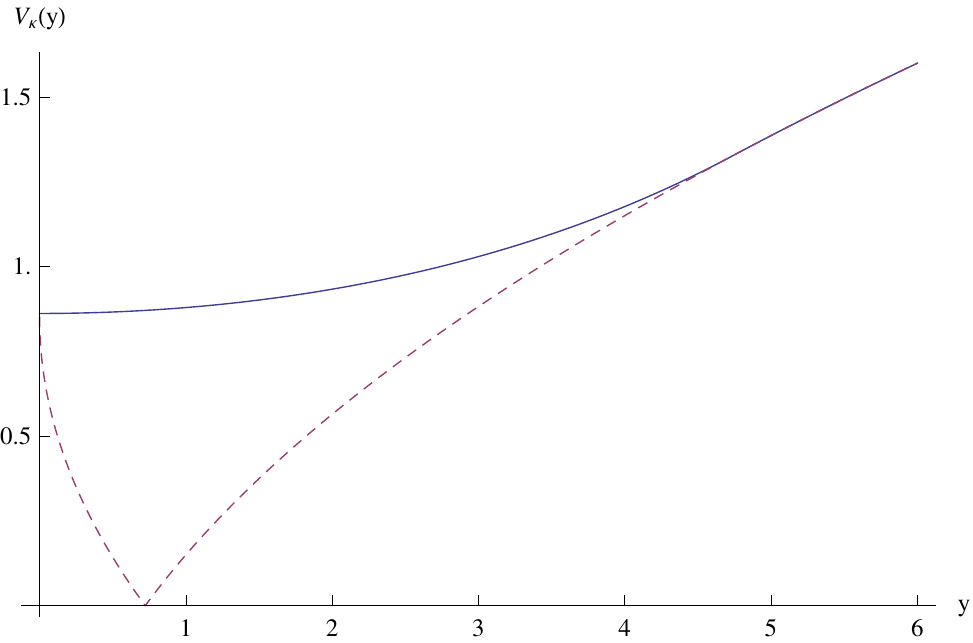}
\caption{\small The value (uniform) and exercise payoff (dashed)}\label{Bessel2}
\end{center}
\end{figure}

\subsection{A truly two-dimensional modification}
The explicit solvability of the problem described before is based on the dimension reduction due to the symmetry of the situation. Even when slightly breaking this symmetry, there is usually no hope to find such explicit solutions anymore. In the rest of this section, we will illustrate this by an example and show how these more general problems may be treated. We consider again the radially symmetric payoffs \eqref{radialpayoff}, but instead of assuming $\|{\bm \theta}_t \|^2\leq \kappa^2$  for the density process, we now assume that
\[\|{\bm \theta}_t \|_\infty\leq \kappa,\mbox{ i.e., }\max\{|\theta_{1t}|,|\theta_{2t}|\}\leq \kappa\]
and denote the set of all corresponding probability measures by $\hat{\mathcal P}^\kappa.$ We note that this ambiguity structure has been considered in \cite{AlCh2019}. We write
\begin{align*}
\hat V_\kappa(\bm x) = \sup_{\tau\in \mathcal{T}}\inf_{\mathbb{Q}^{{\bm \theta}}\in \hat{\mathcal{P}}^\kappa}\mathbb{E}_{\bm x}^{{\mathbb{Q}^{{\bm \theta}}}}\left[e^{-r\tau}F(\mathbf{X}_\tau)\mathbbm{1}_{\{\tau < \infty\}}\right]
\end{align*}
and $\hat C_\kappa$ for the corresponding continuation set. As $\frac{1}{\sqrt{d}}\|\cdot\|\leq \|\cdot\|_\infty\leq \|\cdot\|$, it is clear that for all $\bm x$
\[V_{\kappa \sqrt{d}}(\bm x)\leq \hat V_\kappa(\bm x)\leq V_{\kappa}(\bm x) \]
and therefore
\[C_{\kappa\sqrt{d}}\subseteq \hat C_{\kappa}\subseteq C_{\kappa}.\]
For the sake of simplicity, we now restrict our attention to the case $d=2$ and $F(y)=y$. In this case, it is -- using the results of this section  -- easily seen that $C_{\kappa\sqrt{d}}$ and $C_\kappa$ are circles around 0. Although there is little hope for finding $\hat V_\kappa$ and $\hat C_{\kappa}$ explicitly, it is easy to infer the structure of the solution: The worst case measure is characterized by the density generator
\[\hat{\bm{\theta}}^\ast=(\kappa \sgn(x_{1}),\kappa \sgn(x_{2})).\]
Due to symmetry of the situation, the optimal stopping problem to be solved can be written as
\[\hat V_\kappa({\bf x})=\sup_{\tau\in \mathcal{T}}\mathbb{E}_{\bm x}^{\hat{\bm{\theta}}^\ast}\left[e^{-r\tau}F(\mathbf{X}_\tau)\mathbbm{1}_{\{\tau < \infty\}}\right],\]
for ${\bf x}$ is in the upper quadrant $\R_+^2$, where $X$ is a Brownian motion with drift $(-\kappa,-\kappa)$ and (orthogonal) reflection on the boundaries of $\R_+^2$. Note that reflected Brownian motion in the quadrant were studied extensively, see \cite{MR628959,MR799421} to mention just two. Recently, the Green kernel has been found semi-explicitly (in the transient case), see \cite{Franceschi:aa}. This opens the door to characterize the unknown optimal stopping boundary using integral equation techniques, see \cite{MR2256030} for the general theory and \cite{MR3958442} for a specific setting quite close to this one.

\section{Conclusions}

We analyzed the impact of Knightian uncertainty on the optimal timing policy of an ambiguity averse decision maker in the case where the underlying follows a multidimensional Brownian motion. We identified two special cases under which the problem can be explicitly solved and illustrated our findings in explicitly parameterized examples. Our results indicate that Knightian uncertainty does not only accelerate the optimal timing policy in comparison with the unambiguous benchmark case, it also may add stability to the dynamics of the underlying under the worst case measure. More precisely, even thought the underlying multidimensional Brownian motion does not converge in a long run to a stationary distribution, the controlled process does. This observation shows that ambiguity may in some circumstances have a profound and nontrivial impact on the underlying dynamics.

This study modeled the underlying random factor dynamics as a multidimensional Brownian motion and focused on two functional forms permitting the utilization of dimension reduction techniques and in that way resulting into stopping problems of linear diffusions. There is at least three natural directions towards which our chosen modeling framework could be attempted to be extended. First, even though most standard factor models rely on linear combinations of the driving factors, it would naturally be of interest to analyze how potential nonlinearities would affect the optimal timing decision in the presence of ambiguity. Especially, introducing state-dependent factors would cast light on the mechanisms how nonlinearities in factor dynamics affect the decisions of ambiguity averse decision makers. Second, carrying out a thorough analysis of the truly two-dimensional modification presented in subsection 4.2 would provide valuable information on the difference between the problems allowing dimensionality reduction and the problems which do not. Third, adding Bayesian learning to the considered class of problems would also be an interesting direction towards which our analysis could be extended. All these extensions are extremely challenging and at the present outside the scope of the this study.

\bibliographystyle{apalike}
\bibliography{radiallysymmetric}

\end{document}